\documentclass[a4paper]{amsart}

\usepackage{amsmath,amssymb,amsthm}
\usepackage{indentfirst,color}
\usepackage[colorlinks,citecolor=red,linkcolor=blue,urlcolor=cyan]{hyperref}

\newcounter{thmletter} 

\theoremstyle{plain}
\newtheorem{theorem}{Theorem}[section]
\newtheorem{proposition}[theorem]{Proposition}
\newtheorem{lemma}[theorem]{Lemma}
\newtheorem{corollary}[theorem]{Corollary}
\newtheorem{ques}[theorem]{Question}
\theoremstyle{definition}
\newtheorem{definition}[theorem]{Definition}
\theoremstyle{remark}
\newtheorem{remark}[theorem]{Remark}
\numberwithin{equation}{section}

\newcommand{\be}{\begin{equation}}
\newcommand{\ee}{\end{equation}}
\newcommand{\bea}{\begin{eqnarray}}
\newcommand{\eea}{\end{eqnarray}}
\newcommand{\ben}{\begin{eqnarray*}}
	\newcommand{\een}{\end{eqnarray*}}
\newcommand{\bt}{\begin{split}}
	\newcommand{\et}{\end{split}}
\newcommand{\bet}{\begin{equation}}
\newcommand{\BB}{\mathbb{B}}
\newcommand{\CC}{\mathbb{C}}

\newcommand{\beq}{\begin{equation*}}
\newcommand{\eeq}{\end{equation*}}
\newcommand{\bal}{\begin{aligned}}
\newcommand{\eal}{\end{aligned}}

\newcommand{\ddbar}{\partial \bar{\partial}}
\newcommand{\dbar}{\bar{\partial}}

\newcommand{\calE}{{\mathcal{E}}}%

\newcommand{\calO}{{\mathcal{O}}}%
\newcommand{\supp}{{\textup{Supp}\,}}%

\newcommand{\loc}{{\textup{loc}}}%

\newcommand{\Dom}{{\textup{Dom}}}

\newcommand{\eps}{\varepsilon}%
\newcommand{\inner}[1]{\langle#1\rangle}

\begin{document}

\title[Convolution Regularization Preserves the $L^2$-Estimate Property]{Convolution Regularization Preserves the $L^2$-Estimate Property for $(1,n)$-Forms}

\author{Zhuo Liu}
\address{Mathematical Science Research Center,  Chongqing University of Technology,  No. 69, Hongguang
Avenue, Banan District, Chongqing 400054, China.}
\email{liuzhuo@cqut.edu.cn; liuzhuo@amss.ac.cn}

\date{}

\begin{abstract}
In this paper, we prove that the \(L^2\)-estimate  property  for \((1,n)\)-forms is preserved under the standard convolution regularization. As applications, we show that any singular Hermitian metric satisfying the optimal or multiple coarse \(L^2\)-estimate  property for \((n,1)\) or \((1,n)\)-forms is Griffiths semi-positive. This resolves a question posed by Deng--Ning--Wang and a question by Inayama.
\end{abstract}



\maketitle

\section{Introduction}

Plurisubharmonicity and Griffiths/Nakano positivity play fundamental roles in several complex variables and complex geometry, yielding numerous important results, including H\"ormander's $L^2$ estimates for the $\overline{\partial}$-equations (\cite{Dem-82, Hor65}), Ohsawa-Takegoshi's $L^2$-extension theorem (\cite{OT87}),  the optimal $L^2$ extension theorem (\cite{Blocki13, GuanZhou12, GuanZhou15}) and so on. Plurisubharmonic functions, which need not be smooth, offer significant advantages in addressing many problems. Similarly, there has been growing interest in singular metrics with positivity properties for vector bundles.​ A singular Hermitian metric $h$ on a holomorphic vector bundle $E$  is a measurable map from the base manifold $X$ to the space of non-negative Hermitian forms on the fibers  satisfying $0<\det h<+\infty$ almost everywhere. The metric $h$	is called \textit{Griffiths semi-positive}  if $|u|^2_{h^*}$ is plurisubharmonic for any local holomorphic section $u$ of the dual bundle $E^*$ (\cite{BP08, PT18, Rau15}). For line bundles, Griffiths semi-positivity of ​$h$ is equivalent to the plurisubharmonicity of the local weight  $\varphi := -\log h$.

Recently, in \cite{DNW21, DNWZ22, DWZZ18}, Deng, Ning, Wang, Zhang, Zhou established the converse $L^2$ theory by giving alternative characterizations of plurisubharmonicity and Griffiths/Nakano positivity in terms of various $L^2$-conditions for $\overline{\partial}$. These characterizations are significant precisely because they avoid metric differentiations, thereby offering workable definitions of various positivity notions for singular Hermitian metrics on holomorphic vector bundles. For more results on characterizations of positivity, we refer the reader  to \cite{DNW21, DNWZ22, DWZZ18,  DengZhang21, HosonoInayama, LMNZ24, Watanabe} and references therein.

\begin{definition}{\cite{DNW21, DNWZ22, Watanabe}}
Let $(E,h)$ be a trivial holomorphic vector bundle over a domain $D \subset \mathbb{C}^n$, equipped with a singular Hermitian metric $h$.

We say that $(E,h)$ satisfies the \textbf{coarse $L^2$-estimate property with constant $C$ for $(p,q)$-forms} if there exists a uniform constant $C>0$ such that for every Stein domain $U\Subset D$, every complete K\"ahler metric $\omega$ on $U$, every smooth strictly plurisubharmonic function $\phi$ on $U$, and every $\dbar$-closed $E$-valued $(p,q)$-form $f$ on $U$ with finite right-hand side, there is $u\in L^2_{(p,q-1)}(U,E,he^{-\phi})$ such that $\dbar u = f$ and
\begin{equation}\label{eq:a1}
  \int_U |u|^2_{\omega,h} e^{-\phi}\,dV_\omega
\le
C \int_U \langle B^{-1}_{\omega,\phi} f,f \rangle_{\omega,h} e^{-\phi}\,dV_\omega,
\end{equation}

where $B_{\omega,\phi}=[i\partial\bar\partial\phi \otimes \mathrm{Id}_E,\, \Lambda_\omega]$.

We say that $(E,h)$ satisfies the \textbf{optimal $L^2$-estimate property}   for $(p,q)$-forms or  \textbf{ $(p,q)$-$L^2$-optimal} if the above holds with $C=1$.

We say that $(E,h)$ satisfies the \textbf{multiple coarse $L^2$-estimate property} for $(p,q)$-forms if for each $m\ge 1$, the tensor power $(E^{\otimes m}, h^{\otimes m})$ satisfies the coarse $L^2$-estimate property for $(p,q)$-forms with constant $C_m$, and
\(
\lim_{m\to\infty} m^{-1}{\log C_m} = 0.
\)
\end{definition}

\begin{remark}\label{rmk def}
  \begin{enumerate}
    \item There is a minor difference from the definition given by Deng--Ning--Wang--Zhou. We require the condition to hold for every complete K\"ahler metric $\omega$ on $U$. Indeed, in the $(n,1)$ setting, the integrals in \eqref{eq:a1} are independent of the choice of $\omega$, since $u$ is an $(n,0)$-form and $f$ is an $(n,1)$-form.
    \item  Let \(U, \omega, E, h, \phi\) be as above. Suppose that \(h\) is smooth and that both \([i\Theta_{(E,h)}, \Lambda_\omega]\) and \(B_{\omega,\phi}\) are semi-positive on \((p,q)\)-forms. Then the \(\bar\partial\)-equation is solvable with optimal \(L^2\)-estimates (see \cite[Chapter VIII-(4.5) Theorem]{D12a}).
    \item Deng--Ning--Wang--Zhou (\cite[Theorem 1.1]{DNWZ22}) proved that a $C^2$-smooth Hermitian metric is Nakano semi-positive if and only if it satisfies the {optimal $L^2$-estimate property} for $(n,1)$-forms. This characterization leads to a definition of Nakano positivity for singular Hermitian metrics (see \cite{Ina-AG}) and provides a positive answer to a question of Lempert (see \cite{LiuYangZhou21, LXYZ24}).
    \item Deng--Ning--Wang--Zhou ( \cite[Theorem 1.2]{DNWZ22}) showed that  a continuous Hermitian metric is Griffiths semi-positive if  it satisfies the  {multiple coarse $L^2$-estimate property} for $(n,1)$-forms.
    \item Watanabe (\cite{Watanabe}) proved that for a \(C^2\)-smooth Hermitian metric \(h\), the operator \([i\Theta_{(E,h)}, \Lambda_\omega]\) is semi-positive on \((p,n)\) or \((n,q)\)-forms if and only if \(h\) satisfies the optimal \(L^2\)-estimate property for \((p,n)\) or \((n,q)\)-forms. In particular, \(h\) is dual Nakano semi-positive if and only if it satisfies the optimal \(L^2\)-estimate property for \((1,n)\)-forms.
  \end{enumerate}
\end{remark}

In \cite[Remark 1.9]{DNW21}, Deng--Ning--Wang asked whether a continuous function $\varphi$ with $e^{-\varphi}$ satisfying the optimal $L^2$-estimate property for $(n,1)$-forms must be plurisubharmonic. In addition, Inayama \cite[Question 7.3]{Ina-AG} raised the question of whether the optimal $L^2$-estimate property for $(n,1)$-forms implies Griffiths positivity for singular Hermitian metrics. 

To state the question in full generality, it is necessary to consider appropriate regularity assumptions. Recall that plurisubharmonic functions are \textit{strongly upper semi-continuous}, i.e.,
\(
\limsup_{S\not\ni z\to x} \varphi(z)=\varphi(x)
\)
for any $x\in D$ and any measure-zero subset $S\subset D$. Moreover, Berndtsson--P\u{a}un \cite[Proposition 3.1]{BP08} showed that a Griffiths semi-negative singular Hermitian metric can be locally approximated from above by a decreasing sequence of smooth Griffiths semi-negative Hermitian metrics via convolution. Motivated by these facts, we unify the above questions as follows:

\begin{ques}[{\cite[Remark 1.9]{DNW21},\cite[Question 7.3]{Ina-AG}}]\label{conj DNW-Ina}
 Let $h$ be a singular Hermitian metric on a holomorphic vector bundle $E$. Assume that the norm $|s|^2_{h^*}$ is \textbf{strongly upper semi-continuous} for every holomorphic local section $s$ of $E^*$. Suppose further that $h$ is  \textbf{lower semi-continuous}, i.e., locally there exists a sequence of continuous Hermitian metrics increasing to $h$. If $h$ satisfies the  optimal or multiple coarse $L^2$-estimate property for $(1,n)$ or $(n,1)$-forms, then is $h$ necessarily Griffiths semi-positive?
\end{ques}

In this paper, we establish the preservation of the coarse \(L^2\)-estimate  property for \((1,n)\)-forms under convolution.

\begin{theorem}  \label{MT1}
Let \(E\) be a holomorphic vector bundle over a complex manifold \(X\), endowed with a lower semi-continuous Hermitian metric \(h\). If \(h\) satisfies the coarse \(L^2\)-estimate property with constant \(C\) for \((1,n)\)-forms, then locally there exists a sequence of smooth Hermitian metrics \(h_\delta\) which also satisfies the coarse \(L^2\)-estimate property with the same constant \(C\) for \((1,n)\)-forms and approximates \(h\) in \(L^1_{\mathrm{loc}}\).
\end{theorem}

\begin{remark}
\begin{enumerate}
  \item In fact, the approximating sequence \(h_\delta\) is obtained by the standard convolution technique applied to the dual metric, as introduced by Berndtsson--P\u{a}un \cite[Proposition 3.1]{BP08}, where they showed that \(h_\delta\) is Griffiths semi-positive whenever \(h\) is Griffiths semi-positive.
  \item Additionally, Raufi \cite[Proposition 1.9]{Rau15} proved that if \(h\) is continuous and dual Nakano semi-positive in the sense of Berndtsson (see \cite[Section 2]{Bern09} and \cite[Definition 1.8]{Rau15}), then \(h_\delta\) is also dual Nakano semi-positive in the same sense.
  \item Recently, Deng--Ning--Wang--Zhou   \cite[Theorem 1.1]{DNWZ23} showed that the aforementioned results of Berndtsson--P\u{a}un and Raufi  are also valid for  all relatively compact Stein open subsets \(U \subset X\).
\end{enumerate}
\end{remark}

Then by Remark \ref{rmk def}-(4), we obtain that

\begin{corollary}
Let \(E\) be a holomorphic vector bundle over a complex manifold \(X\), endowed with a lower semi-continuous Hermitian metric \(h\). If \(h\) satisfies the optimal \(L^2\)-estimate property for \((1,n)\)-forms, then locally there exists a sequence of smooth dual Nakano semi-positive Hermitian metrics \(h_\delta\) approximating \(h\) in \(L^1_{\mathrm{loc}}\).
\end{corollary}

To address Question \ref{conj DNW-Ina}, recall the following key facts:
\begin{itemize}
  \item Griffiths positivity is characterized by restriction to complex lines (Lemma \ref{lem ch Grif line});
  \item For smooth Hermitian metrics with \(\dim X=1\), Deng--Ning--Wang--Zhou proved that the optimal or multiple coarse \(L^2\)-estimate property for \((1,1)\)-forms implies Griffiths positivity (Lemma \ref{lem DNWZ Grif});
  \item Under certain regularity assumptions on \(h\), the coarse \(L^2\)-estimate property for \((n,1)\) or \((1,n)\)-forms of singular Hermitian metrics is preserved under restriction to complex hyperplanes (Propositions \ref{pro nak res nak} and \ref{pro dnak res dnak}).
\end{itemize}

If \(h\) satisfies the optimal or multiple coarse \(L^2\)-estimate property for \((1,n)\)-forms, then by Theorem \ref{MT1} there exists a local smooth approximation preserving the estimate. By Proposition \ref{pro dnak res dnak}, the restriction to any complex line still satisfies the \((1,1)\)-estimate, so Deng--Ning--Wang--Zhou's result yields Griffiths positivity on each line. If instead \(h\) satisfies the \((n,1)\)-estimate, then by Proposition \ref{pro nak res nak}, after fixing a direction, the restriction to almost every parallel complex line satisfies the \((1,1)\)-estimate. Applying Theorem \ref{MT1} again gives a smooth approximation preserving the estimate, so these restrictions are Griffiths semi-positive by Deng--Ning--Wang--Zhou. Finally, the strong upper semi-continuity assumption extends the conclusion to every complex line. Therefore, by Lemma \ref{lem ch Grif line},  we give an affirmative answer to Question \ref{conj DNW-Ina}.

\begin{theorem}
Let \((E,h)\) be a holomorphic vector bundle on a complex manifold, equipped with a lower semi-continuous Hermitian metric \(h\). Assume that the norm \(|s|^2_{h^*}\) is strongly upper semi-continuous for every holomorphic local section \(s\) of \(E^*\). If \(h\) satisfies the optimal or multiple coarse \(L^2\)-estimate property for \((n,1)\) or \((1,n)\)-forms, then \(h\) is Griffiths semi-positive.
\end{theorem}

In particular, we obtain the following characterization of plurisubharmonic functions, which relaxes the continuity assumption in \cite[Theorem 1.5]{DNW21} to strong upper semi-continuity.

\begin{corollary}
A measurable function \(\varphi:D\subset\mathbb C^n\to [-\infty,+\infty)\) is plurisubharmonic if and only if \(\varphi\) is strongly upper semi-continuous and \((D\times\mathbb C, e^{-\varphi})\) satisfies the multiple coarse \(L^2\)-estimate property for \((n,1)\) or \((1,n)\)-forms.
\end{corollary}

The remaining parts of this article are organized as follows. Section \ref{ss Lebesgue} presents a Lebesgue-type differentiation theorem; Section \ref{ss Res} proves the preservation of the coarse \(L^2\)-estimate property under restriction to complex hyperplanes; Section \ref{ss dNak} establishes the preservation of the coarse \(L^2\)-estimate property for \((1,n)\)-forms under convolution; and Section \ref{ss solution to conj DNW-Ina} gives the positive answer to Question \ref{conj DNW-Ina}.

\textbf{Acknowledgement:}
The initial proof attempt of Theorem \ref{MT1} was inspired by a preliminary and incomplete argument for bounded planar domains generated by AI. By examining this defective preliminary work, I was naturally led to the related results of Raufi \cite{Rau15} and Watanabe \cite{Watanabe}, which motivated me to generalize the argument to higher-dimensional settings. This observation enables us to establish the preservation of the coarse \(L^2\)-estimate property for \((1,n)\)-forms under convolution operations. Combined with the previous results in \cite{Liu25b+}, we finally obtain an affirmative answer to Question \ref{conj DNW-Ina}. The author was supported by the National Natural Science Foundation of China (No.12501101) and the Scientific Research Foundation of Chongqing University of Technology  (No.2025ZDZ013).

\section{Preliminaries}

\subsection{Lebesgue-type differentiation theorem}\label{ss Lebesgue}

In this subsection, we prove a Lebesgue-type differentiation theorem for hyperplanes, following the proof of the classical Lebesgue differentiation theorem.

\begin{lemma}[Basic covering lemma, {\cite[Theorem 1.2]{Hein01}}]\label{lem covering}
  Every family $\mathcal{F}$ of balls of uniformly bounded diameter in $\CC^n$ contains a disjointed subfamily $\mathcal{G}$, which means that no two balls in $\mathcal{G}$ meet, such that
  $$\bigcup_{B\in\mathcal{F}}B\subset\bigcup_{B\in\mathcal{G}} 5B.$$
\end{lemma}

\begin{theorem}[Lebesgue differentiation theorem]
  Let $f\in L^1(\CC^n;\loc)$, then for almost all $z\in\CC^n$, $$\lim_{\eps\to 0}\frac{1}{\lambda_n(\BB^n_z(\eps))}\int_{\BB^n_z(\eps)}
  |f(w)-f(z)|d\lambda_n(w)=0.$$
\end{theorem}

Let $f\in L^1(\CC^n)$. It follows from  Fubini's theorem that the function $$F(z''):=\int_{\CC^{n-k}}f(z',z'')d\lambda_{n-k}(z')$$ is locally integrable on $\CC^k$. Then by the Lebesgue differentiation theorem, we have

\begin{proposition}\label{prop weak Lebesgue hyperplane}
  Let $f\in L^1(\CC^n)$, then for almost all $z''\in\CC^k$, $$\lim_{\eps\to 0} \int_{\BB^k_{z''}(\eps)}\hspace{-3.05em}-\quad\quad
  \left|\int_{\CC^{n-k}}f(w',w'')-f(w',z'')d\lambda_{n-k}(w')\right|
  d\lambda_k(w'')=0.$$
\end{proposition}

\begin{remark}
  The above result is not enough to  prove Proposition \ref{pro nak res nak} and  Proposition \ref{pro dnak res dnak}. Fortunately, by repeating the proof of the Lebesgue differentiation theorem, we obtain a strong form of the above proposition.
\end{remark}

\begin{definition}
  Let $f\in L^1(\CC^n)$. Define the maximal function $$Mf(z''):=\sup_{\eps> 0} \int_{\BB^k_{z''}(\eps)}\hspace{-3.05em}-\quad\quad
  \int_{\CC^{n-k}}|f(w',w'')|d\lambda_{n-k}(w')
  d\lambda_k(w'')$$
\end{definition}

\begin{proposition}
   Let $f\in L^1(\CC^n)$, then the maximal function $Mf$ is lower semi-continuous on $\CC$.
\end{proposition}
\begin{proof}
It follows from Fubini's theorem that $Mf$ is finite almost everywhere.
  For any $t\in(0,+\infty)$, it suffices to show that $E_t:=\{z''\in \CC^k; Mf(z'')>t\}$ is open. Fix $z''_0\in E_t$, then there is $\eps>0$ such that $$ \int_{\BB^k_{z''_0}(\eps)}\hspace{-3.05em}-\quad\quad\int_{
  \CC^{n-k}}
   |f(w',w'')|d\lambda_{n-k}(w')d\lambda_k(w'')>t.$$
   Take $\eps'>\eps$ with $$ \int_{\BB^k_{z''_0}(\eps)}\hspace{-3.05em}-\quad\quad\int_{
  \CC^{n-k}}
   |f(w',w'')|d\lambda_{n-k}(w')d\lambda_k(w'')>t,$$ then for any $z''\in \CC$ with $|z''-z''_0|<\eps'-\eps$, we have $\BB^k_{z''_0}(\eps)\subset\BB^k_{z''}(\eps')$ and \begin{align*}
     t <& \int_{\BB^k_{z''_0}(\eps)}\hspace{-3.05em}-\quad\quad\int_{
  \CC^{n-k}}
   |f(w',w'')|d\lambda_{n-k}(w')d\lambda_k(w'')\\
      \le&  \int_{\BB^k_{z''}(\eps')}\hspace{-3.3em}-\quad\quad\ \int_{
  \CC^{n-k}}
   |f(w',w'')|d\lambda_{n-k}(w')d\lambda_k(w'')\le Mf(z'').
   \end{align*}
   Hence $\BB^k_{z''_0}(\eps'-\eps)\subset E_t$ and we complete the proof.
   \end{proof}

\begin{theorem}[Hardy-Littlewood theorem]
  Let $f\in L^1(\CC^n)$, then for any $t\in(0,+\infty)$, $$\lambda_k(\{z''\in\CC^k; Mf(z'')>t\})\le 5^{2k}t^{-1}\|f\|_{L^1}.$$
\end{theorem}
\begin{proof}
  Fix $t>0$ and denote $E_t:=\{z''\in\CC^k; Mf(z'')>t\}$.  It suffices to show that  any compact subset $K$ of $E_t$ satisfies $$\lambda_k(K)\le 5^{2k}t^{-1}\|f\|_{L^1}.$$

  For $z''\in K$, there is open disk $\BB_{z''}(\eps_{z''})$ such that  $$ \int_{\BB^{k}_{z''}(\eps_{z''})}\hspace{-3.9em}-\quad\quad\quad\int_{
  \CC^{n-k}}
   |f(w',w'')|d\lambda_{n-k}(w')d\lambda_k(w'')>t.$$
   Since $K$ is compact, then there are finite open balls $\{\BB_{z''_j}(\eps_{z''_j}); 1\le j\le N\}$ covering $K$. It follows from the basic covering lemma that there are mutually disjoint open balls $\BB_{z''_{j_\ell}}(\eps_{z''_{j_\ell}}), 1\le \ell\le M (M\le N)$   such that $\{\BB_{z''_{j_\ell}}(3\eps_{z''_{j_\ell}}); 1\le \ell\le M\}$ covering $K$. Hence we have
   \begin{align*}
     \lambda_k(K) \le& 5^{2k}\sum_{\ell=1}^{M}\lambda_k(\BB_{z''_{j_\ell}}(\eps_{z''_{j_\ell}})) \\
      \le & 5^{2k}t^{-1} \sum_{j=1}^{M}\int_{\BB_{z''_{j_\ell}}(\eps_{z''_{j_\ell}})}\int_{
  \CC^{n-k}}
   |f(w',w'')|d\lambda_{n-k}(w')d\lambda_k(w'')\\
   \le& 5^{2k}t^{-1} \|f\|_{L^1}.
   \end{align*}
\end{proof}

\begin{proposition}\label{prop f*}
 Let $f\in L^1(\CC^n)$ and  define $$f^*(z''):=\limsup_{\eps\to 0} \int_{\BB^k_{z''}(\eps)}\hspace{-3.05em}-\quad\quad
  \int_{\CC^{n-k}}|f(w',w'')-f(w',z'')|d\lambda_{n-k}(w')
  d\lambda_k(w'').$$ Then we have
  \begin{enumerate}
    \item[(i)] If $f\in C^0_c(\CC^n)$, then $f^*=0$.
    \item[(ii)] If $g\in L^1(\CC^n)$ with $g^*=0$, then $f^*=(f-g)^*$.
    \item[(iii)] $\lambda_k\{z''\in\CC^k;f^*(z'')>t\}\le 2(5^{2k}+1)t^{-1}\|f\|_{L^1}.$
  \end{enumerate}
\end{proposition}
  \begin{proof}
  \begin{enumerate}
    \item[(i)] Since $f\in C^0_c(\CC^n)$, then $\supp(f)\subset \BB^{n-k}_0(R)\times \BB^{k}_0(R)$ for some $R>0$.  In addition,  for any $\delta>0$, there is $r>0$ such that $$|f(w',w'')-f(w',z'')|<\delta$$ as long as $|w''-z''|<r$.  Hence for $\eps<r$, we have \begin{align*}
                                                 \int_{\BB^k_{z''}(\eps)}\hspace{-3.05em}-\quad\quad
  \int_{\CC^{n-k}}|f(w',w'')-f(w',z'')|d\lambda_{n-k}(w')
  d\lambda_k(w'')
                                              \le &  \lambda_{n-k}(\BB^{n-k}_0(R))\delta.
                                            \end{align*}  As a result, $f^*=0$.
    \item[(ii)] By the  triangle inequality, we have  \begin{align*}
               (f-g)^* \le& f^*+g^*= f^* = (f-g+g)^*\le (f-g)^*+g^*=(f-g)^*.
             \end{align*}
    \item[(iii)] Since \begin{align*}
                         f^*(z'') \le & \sup_{\eps\to 0} \int_{\BB^k_{z''}(\eps)}\hspace{-3.05em}-\quad\quad
  \int_{\CC^{n-k}}|f(w',w'')-f(w',z'')|d\lambda_{n-k}(w')
  d\lambda_k(w'') \\
                         \le & Mf(z'')+ \int_{\CC^{n-k}}|f(w',z'')|d\lambda_{n-k}(w').
                       \end{align*}
                       Note that  $F(z''):=\int_{\CC^{n-k}}|f(w',z'')|d\lambda_{n-k}(w')$ is integrable on $\CC^k$, then by
                       Chebyshev's inequality, we have $$\lambda_k\{z''\in\CC^k;F(z'')>t\}\le t^{-1}\|F\|_{L^1}=t^{-1}\|f\|_{L^1}.$$
      Therefore, \begin{align*}
               & \lambda_k\{z''\in\CC^k;f^*(z'')>t\}\\ \le    & \lambda_k\{z''\in\CC^k;Mf^*(z'')>\frac{t}{2}\}+\lambda_k\{z''\in\CC^k;F(z'')>\frac{t}{2}\}\\
                   \le & 2(5^{2k}+1)t^{-1}\|f\|_{L^1}.
                 \end{align*}
  \end{enumerate}

  \end{proof}

Now we can strength the Proposition \ref{prop weak Lebesgue hyperplane} as follows.

\begin{proposition}\label{prop Lebesgue hyperplane}
  Let $f\in L^1(\CC^n)$, then for almost all $z''\in\CC^k$, $$\int_{\CC^{n-k}}| f(w',z'')|d\lambda_{n-k}(w')<+\infty$$ and $$\lim_{\eps\to 0} \int_{\BB^k_{z''}(\eps)}\hspace{-3.05em}-\quad\quad
  \int_{\CC^{n-k}}|f(w',w'')-f(w',z'')|d\lambda_{n-k}(w')
  d\lambda_k(w'')=0.$$
\end{proposition}

\begin{remark}
  Let $U\subset \CC^{n-k}$, $D\subset\CC^k$ and $f\in L^1(U\times D)$. Consider
  \begin{equation*}
   \tilde{f}(z):=\left\{\begin{aligned}
    &f(z), &z\in U\times D;\\
    &0,    &z\notin U\times D.
  \end{aligned}\right.
    \end{equation*}
Then $\tilde{f}\in L^1(\CC^n)$ and Proposition \ref{prop Lebesgue hyperplane} implies that for almost all $z''\in D$, $$\lim_{\eps\to 0} \int_{\BB^k_{z''}(\eps)}\hspace{-3.05em}-\quad\quad
  \int_{U}|f(w',w'')-f(w',z'')|d\lambda_{n-k}(w')
  d\lambda_k(w'')=0.$$
\end{remark}

\begin{proof}
It suffices to show that $f^*=0$ almost everywhere.
 Since $f\in L^1(\CC^n)$, for any $\delta>0$, we can find a function $g_\delta\in C^0_c(\CC^n)$ such that $\|f-g_\delta\|_{L^1}\le\delta.$ Then it follows from Proposition \ref{prop f*} that
 \begin{align*}
   \lambda_k\{z''\in\CC^k;f^*(z'')>t\}= &\lambda_k\{z''\in\CC^k;(f-g_\delta)^*(z'')>t\}  \\
   \le & 2(5^{2k}+1)t^{-1}\|f-g_\delta\|_{L^1} \\
    \le &  2(5^{2k}+1)t^{-1}\delta.
 \end{align*}
 Since $\delta$ is arbitrary, we get $\lambda_k\{z''\in\CC^k;f^*(z'')>t\}=0$ for any $t>0$. We complete the proof.

\end{proof}

\subsection{Preservation of the coarse\(L^2\)-estimate property for \((n,1)\) or \((1,n)\)-forms  restricted on hyperplanes}
\label{ss Res}

In this subsection, we establish the preservation of the coarse \(L^2\)-estimate property for \((n,1)\) or \((1,n)\)-forms under restriction to hyperplanes. The proof follows \cite[Proposition 3.3]{LXYZ24}, with the generalized Siu lemma \cite{ZZsiulem} replaced by a Lebesgue-type differentiation theorem (Lemma \ref{lem weak siu lemma}).

\begin{definition}[\cite{Cataldo98}]
    The multiplier submodule sheaf $\mathcal{E}(h)$ of $\mathcal{O}(E)$ associated to $(E,h)$  is the sheaf of the germs $s_x \in \mathcal{O}(E)_x$
    such that $|s_x|^2_h$ is integrable in some neighborhood of $x$.
\end{definition}

\begin{lemma}[\cite{HosonoInayama}]\label{lem coherent}
Let $(E,h)$ be a holomorphic vector bundle over a  domain $D \subset \mathbb{C}^n$ equipped with a {singular Hermitian metric} $h$. If $h^*$ is  upper semi-continuous and satisfies the  coarse $L^2$-estimate property, then the  multiplier submodule sheaf $\mathcal{E}(h) $  is coherent. In particular,  the support set of the quotient sheaf $\mathcal{O}(E)/\mathcal{E}(h)$  is an analytic subset of $D$.
\end{lemma}

\begin{lemma}\label{lem weak siu lemma}
  Let $h$ be an lower semi-continuous Hermitian metric on a trivial holomorphic vector bundle $E$ of rank $r$ over a bounded domain $U\times D\subset \CC^{n-k}_{w'}\times\CC^k_{w''}$. Assume that $\calE(h)=\calO(E)$. Then almost all $z''\in D$ satisfy that  for any  relative compact subset $V\Subset U$  any continuous section $f\in\Gamma(\overline{V}\times D,E)$ and any continuous Hermitian positive matrix function $P\in   C^0(\overline{V}\times D,\rm{Herm}^+(E))$, we have $$ \int_{V}\left| \langle Pf,f\rangle_h(w',z'')\right|d\lambda_{n-k}(w')<+\infty$$ and $$\lim_{\eps\to 0} \int_{\BB_{z''}(\eps)}\hspace{-3.05em}-\quad\quad
  \int_{V}\left|\langle Pf,f\rangle_h(w',w'')-\langle Pf,f\rangle_h(w',z'')\right|d\lambda_{n-k}(w')
  d\lambda_n(w'')=0.$$
\end{lemma}
\begin{proof}
Let $\{U_\ell\}$ be a compact exhaustion of $U$, which means $U_\ell\Subset U_{\ell+1}$ for each $\ell$ and $\bigcup_{\ell=1}^{+\infty}U_\ell=U$.  Then for any  relative compact subset $V\Subset U$, $V\subset U_\ell$ for some $\ell$.
  Write $h=(h_{jk})_{1\le j,k\le r}$. Since  $\calE(h)=\calO(E)$, we know that the function $\int_{U_\ell}|h_{jk}(w',w'')|d\lambda_{n-k}(w')$ is locally integrable on $D\subset\CC^k$. Then by Proposition \ref{prop Lebesgue hyperplane}, there is a measurable subset $A_{\ell,j,k}$ of full measure such that  $$\lim_{\eps\to 0}\int_{\BB_{z''}(\eps)}\hspace{-3.05em}-\quad\quad
  \int_{U_\ell}| h_{jk}(w',w'')- h_{jk}(w',z'')|d\lambda_{n-k}(w')
  d\lambda_n(w'')=0$$ for any $z''\in A_{\ell,j,k}$. Therefore,  for any $z''\in\bigcap_{\ell,j,k}A_{\ell,j,k}$, we obtain that $$\lim_{\eps\to 0} \int_{\BB_{z''}(\eps)}\hspace{-3.05em}-\quad\quad
  \int_{V} |h_{jk}(w',w'')- h_{jk}(w',z'')|d\lambda_{n-k}(w')
  d\lambda_n(w'')=0$$ for any $V\Subset U$ and each $1\le j,k\le r$.
  Since \begin{align*}
    &   \int_{V}\left|\langle Pf,f\rangle_h(w',w'')-\langle Pf,f\rangle_h(w',z'')\right|d\lambda_{n-k}(w')   \\
    = &    \int_{V}\left|\sum_{j,k,\ell=1}^{r}P_{j\ell}h_{jk}f_\ell\bar{f}_k(w',w'')- P_{j\ell}h_{jk}f_\ell\bar{f}_k(w',z'')\right|d\lambda_{n-k}(w')   \\
  \le &     \sum_{j,k,\ell=1}^{r}\int_{V}|h_{jk}(w',w'')-h_{jk}(w',z'')||P_{j\ell}f_\ell\bar{f}_k(w',z'')|d\lambda_{n-k}(w')  \\
  & +    \sum_{j,k,\ell=1}^{r}\int_{V}|h_{jk}(w',w'')|\left|P_{j\ell}f_\ell\bar{f}_k(w',w'')-P_{j\ell}f_\ell\bar{f}_k(w',z'')\right|d\lambda_{n-k}(w'),
  \end{align*} and notice that $P$ and $f$ is uniformly continuous on $V\times \BB^k_{z''}(\eps)$, then for any $z''\in\bigcap_{\ell,j,k}A_{\ell,j,k}$, we obtain
  $$\lim_{\eps\to 0} \int_{\BB_{z''}(\eps)}\hspace{-3.05em}-\quad\quad
  \int_{V}\left|\langle Pf,f\rangle_h(w',w'')-\langle Pf,f\rangle_h(w',z'')\right|d\lambda_{n-k}(w')
  d\lambda_n(w'')=0.$$

\end{proof}

\begin{proposition} \label{pro nak res nak}
Let $h$ be a  lower semi-continuous Hermitian metric on a trivial holomorphic vector bundle $E$ over a Stein domain $U\times D\subset \CC_{w'}^{n-k}\times\CC_{w''}^k$. Assume that $\calE(h)=\calO(E)$ and $h$ satisfies the coarse \(L^2\)-estimate property with constant $C$ for \((n,1)\)-forms.  Then for almost all $z''\in D$,  $h|_{U\times\{z''\}}$ also satisfies  the coarse \(L^2\)-estimate property with the same constant $C$  for \((n-k,1)\)-forms.
\end{proposition}
\begin{proof}

Since $U$ is  Stein, there is  a Stein exhaustion $U_\ell$ of $U$ and we have
$\calE(h)=\calO(E)$ on $\overline{U_\ell}\times D$. It follows from Lemma \ref{lem weak siu lemma} that there is a subset $A$ of $D$ such that $\lambda_k(A)=\lambda_k(D)$ and  any $z''\in A$ satisfies that  for any $U_\ell$, any continuous section $f\in\Gamma(\overline{U_\ell}\times D,E)$ and any continuous Hermitian positive matrix function $P\in   C^0(\overline{U_\ell}\times D,\rm{Herm}^+(E))$, we have
$$ \int_{U_\ell}\left| \langle Pf,f\rangle_h(w',z'')\right|d\lambda_{n-k}(w')<+\infty$$ and
\small{\begin{equation}\label{for pro nak res nak aaa}
                                                     \lim_{\eps\to 0} \int_{\BB_{z''}(\eps)}\hspace{-3.05em}-\quad\quad
  \int_{U_\ell}\left|\langle Pf, f\rangle_h(w',w'')-\langle Pf,f\rangle_h(w',z'')\right|d\lambda_{n-k}(w')
  d\lambda_n(w'')=0.
                                                   \end{equation}}

 For any $z''\in A$, let $V\subset U$ be a Stein open subset,  $\omega$  a complete K\"{a}hler metric, $\phi$ a smooth strictly plurisubharmonic function on $V\times\{z''\}$ and $f(w')=\sum_{j=1}^{n-k}f_j(w')d\bar w'_j$ a $\dbar$-closed $E$-valued $(n-k,1)$-form   such that
   \begin{equation*}
   \int_{V\times\{z''\}} \inner{B_{\omega,\phi}^{-1}f,f}_{\omega,h} e^{-\phi} dV_{\omega}<+\infty.
  \end{equation*}
 Denote $V_\ell:=V\cap U_\ell$. Let $f_\ell$ be the smooth approximation sequence of $f$ obtained by convolution on $V_{\ell+1}$, then $\dbar f_\ell=0$ and $f_\ell$ are smooth on $\overline{V_\ell}\times\{z''\}$ with
 \begin{equation}\label{for pro nak res nak bbb}
\lim_{\ell\to+\infty}\int_{V_\ell\times\{z''\}} \inner{B_{\omega,\phi}^{-1}f_\ell,f_\ell}_{\omega,h} e^{-\phi} dV_{\omega}=\int_{V\times\{z''\}} \inner{B_{\omega,\phi}^{-1}f,f}_{\omega,h} e^{-\phi} dV_{\omega}.
 \end{equation}

Let \(\psi\) be a smooth strictly plurisubharmonic exhaustion function on \(\mathbb B^k_{z''}(\eps)\)  such that $\psi(z'')=0$ and \(i\ddbar\psi\) is complete on \(\mathbb B^k_{z''}(\eps)\). Define
\[
\widetilde\omega:=\omega+i\ddbar\psi,\qquad
\phi_\varepsilon(w',w''):=\phi(w')+\varepsilon\,\psi(w'')
\]
on \(V_{\ell,\varepsilon}:=V_\ell\times \mathbb B^k_{z''}(\eps)\).
Then \(\widetilde\omega\) is a complete K\"ahler metric on \(V_{\ell,\varepsilon}\), and \(\phi_\varepsilon\) is a smooth strictly plurisubharmonic function on \(V_{\ell,\varepsilon}\). Moreover,
\[
\widetilde{f_\ell}(w',w'')
:=
\sum_{j=1}^{n-k} f_\ell^{(j)}(w')\wedge dw''\wedge d\bar w'_j
\]
is a \(\bar\partial\)-closed \(E\)-valued \((n,1)\)-form on \(V_{\ell,\varepsilon}\), where \(dw''=dw''_{n-k+1}\wedge\cdots\wedge dw''_n\).

Notice that $\phi_\eps(w',z'')=\phi(w')$ and  $$\langle B^{-1}_{\widetilde\omega,\phi_\eps} \widetilde {f_\ell} ,\widetilde {f_\ell}  \rangle_{\widetilde\omega} =\langle B^{-1}_{\omega,\phi}   f_\ell ,  f_\ell  \rangle_{ \omega}, $$

then by (\ref{for pro nak res nak aaa})  we obtain\small{ \begin{equation}\label{for pro nak res nak ddd}
  \lim_{\eps \to 0} \frac{1}{\lambda_k(\BB^k_{z''}(\eps))}\int_{V_{\ell, \eps}}  \langle B^{-1}_{\widetilde\omega,\phi_\eps} \widetilde {f_\ell} ,\widetilde {f_\ell}  \rangle_{\widetilde\omega,h}e^{-\phi_\eps}dV_{\widetilde\omega}=\int_{V_\ell\times\{z''\}}\langle B^{-1}_{\omega,\phi}   f_\ell,  f_\ell \rangle_{ \omega,h}e^{-\phi}dV_{\omega}.
 \end{equation}}

Since $h$ satisfies the coarse \(L^2\)-estimate property with constant $C$ for \((n,1)\)-forms,
 there is a $v_{\ell,\eps}\in L^2_{n,0}(V_{\ell, \eps},E;\loc)$ such that $\dbar v_{\ell,\eps}=\widetilde f_\ell$ and
     \begin{align*}
    \int_{V_{\ell, \eps}} |v_{\ell,\eps}|^2_{\omega,h} e^{-\phi_\eps} dV_{\widetilde\omega}
    \le C\int_{V_{\ell, \eps}} \langle B^{-1}_{\omega,\phi_\eps}\widetilde {f_\ell},\widetilde {f_\ell} \rangle_{\widetilde\omega,h} e^{-\phi_\eps} dV_{\widetilde\omega}.
 \end{align*}
 Moreover, by the weak regularity of $\dbar$ on $(n,0)$-forms, we can take $v_{\ell,\eps}$ to be smooth.
Using the Fubini-Tonelli theorem, we know that for any $\eps>0$, there exists a $\xi_\eps\in \BB^k_{z''}(\eps)\cap A$ such that \begin{equation*}
            \int_{V_\eps\cap\{w''=\xi_\eps\} }|v_{\ell,\eps}(w',\xi_\eps)|^2_{\widetilde\omega,h}e^{-\phi_\eps}dV_{\widetilde\omega}\le
\frac{1}{\lambda_k(\BB^k_{z''}(\eps))}\int_{V_{\ell, \eps}} |v_{\ell,\eps}|^2_{\widetilde\omega,h} e^{-\phi_\eps} dV_{\widetilde\omega}.
          \end{equation*}
Let $u_{\ell, \eps}(w')=v_{\eps}(w',\xi_\eps)/dw''$, then we have $\dbar u_{\ell, \eps}=f_\ell$ on $U$ and \begin{equation*}
\int_{V_\ell\times\{z''\} }|u_{\ell,\eps}|^2_{\omega,h(w',\xi_\eps)}e^{-\phi-\eps|\xi_\eps-z''|^2}dV_{\omega}\le
\frac{C}{\lambda_k(\BB^k_{z''}(\eps))}\int_{V_{\ell, \eps}} \langle B^{-1}_{\widetilde\omega,\phi_\eps}\widetilde {f_\ell},\widetilde {f_\ell} \rangle_{\widetilde\omega,h} e^{-\phi_\eps} dV_{\widetilde\omega}. \end{equation*}
Since $h$ has positive lower bound locally,
we  can choose a sequence $\eps_m\to 0$ such that $u_{\ell, \eps_m}$  compactly converges to a limit $u_\ell$. Then   $\dbar u=f$. In addition, since $h$ is lower semi-continuous, for any compact subset $K\subset V_\ell$, passing to a finite subcover of
$K$, we may assume there is a sequence of continuous Hermitian metrics $h_j$   on $K$   increasing to $h$. Then by Fatou's lemma, we have
\begin{align*}
   \int_{K\times\{z''\} } |u_\ell|^2_{\omega,h_j}e^{-\phi}dV_\omega    = & \int_{K\times\{z''\} }\lim_{\eps_m\to0}  |u_{\ell,\eps_m}|^2_{\omega,h_j(w',\xi_{\eps_m}) }e^{-\phi-\eps_m|\xi_{\eps_m}-z''|^2}dV_\omega   \\ \le  &\liminf_{\eps_m\to0} \int_{K\times\{z''\} }  |u_{\ell,\eps_m}|^2_{\omega,h_j(w',\xi_{\eps_m}) }e^{-\phi-\eps_m|\xi_{\eps_m}-z''|^2}dV_\omega  \\ \le & \liminf_{\eps_m\to0}  \frac{C}{\lambda_k(\BB^k_{z''}(\eps_m))}\int_{V_{\ell, \eps_m}} \langle B^{-1}_{\widetilde\omega,\phi_{\eps_m}}\widetilde {f_\ell},\widetilde {f_\ell} \rangle_{\widetilde\omega,h} e^{-\phi_{\eps_m}} dV_{\widetilde\omega}\\ \le &  C\int_{V_\ell\times \{z''\} }\langle B^{-1}_{\omega,\phi}   f_\ell,  f_\ell \rangle_{ \omega,h}e^{-\phi}dV_{\omega},
\end{align*}
where the last   inequality is due to (\ref{for pro nak res nak ddd}). Since $h_j$   increases  to $h$ and $K\Subset V_\ell$ is arbitrary, we obtain
$$\int_{V_\ell\times\{z''\} } |u_\ell|^2_{\omega,h}e^{-\phi}dV_\omega \le C \int_{V_\ell\times \{z''\} }\langle B^{-1}_{\omega,\phi}   f_\ell,  f_\ell \rangle_{ \omega,h}e^{-\phi}dV_{\omega}.$$
Then  we can choose a sequence $\ell_m\to 0$ such that $u_{\ell_m}$  compactly converges to a limit $u$ on $V\times\{z''\}$. Then $\dbar u=f$. In addition,  by Fatou's lemma and (\ref{for pro nak res nak bbb}), we have
$$\int_{V\times\{z''\}}|u|^2_{\omega,h}e^{-\phi}dV_\omega\le C\int_{V\times \{z''\} }\langle B^{-1}_{\omega,\phi}   f ,  f  \rangle_{ \omega,h}e^{-\phi}dV_{\omega}.$$

\end{proof}

When \(h\) is smooth, the same argument also works for the coarse \(L^2\)-estimate property with constant \(C\) for \((1,n)\)-forms. Indeed, the computation of the right-hand side integral is identical to the \((n,1)\)-case. However, for general singular Hermitian metrics, the operator \(\bar\partial\) lacks the necessary regularity to ensure that \(v_{\ell,\varepsilon}\) can be taken smooth. In the smooth setting, one may instead take the minimal \(L^2\) solution, whose regularity follows from the standard elliptic regularity of the Laplace operator (see \cite[Chapter VIII-(4.6) Remark]{D12a}); hence the same proof applies.
 Therefore we  obtain the following preservation result for the coarse\(L^2\)-estimate property for \((1,n)\)-forms  under hyperplane restrictions.

\begin{proposition} \label{pro dnak res dnak}
Let $h$ be a smooth Hermitian metric on a trivial holomorphic vector bundle $E$ over a Stein domain $U\times D\subset \CC_{w'}^{n-k}\times\CC_{w''}^k$. Assume that $h$ satisfies the coarse \(L^2\)-estimate property with constant $C$ for \((1,n)\)-forms.  Then for any $z''\in D$,  $h|_{U\times\{z''\}}$ also satisfies  the coarse \(L^2\)-estimate property with the same constant $C$  for \((1,n-k)\)-forms.
\end{proposition}

\section{Preservation of the coarse \(L^2\)-estimate property for \((1,n)\)-forms under convolution}
\label{ss dNak}

In this section, we show that the coarse \(L^2\)-estimate property for \((1,n)\)-forms is stable under convolution.

Let \(D\subset \mathbb C^n\) be a bounded domain and \(E=D\times \mathbb C^r\) a trivial holomorphic vector bundle. Let \(h\) be a lower semi-continuous Hermitian metric on \(E\), with matrix \(h=(h_{\lambda\mu})\) in the trivialization.

For \(E\)-valued forms \(u,v\), define the pointwise Hermitian inner product by
\[
\{u,v\}_h := u^T \wedge h \overline{v}=\sum_{\lambda,\mu}h_{\lambda\mu}u_\lambda\wedge\bar{v}_\mu.
\]
 Similarly, the reference Euclidean inner product is
\[
\{u,v\}_{\mathbb C^r} := u^T \wedge \overline{v}=\sum_{\lambda,\mu}u_\lambda\wedge\bar{v}_\mu.
\]
Therefore,
\begin{equation}\label{eq h}
   \{h^{-T} u, v\}_h = \{u, v\}_{\mathbb C^r}
\quad (u,v\in E_x).
\end{equation}

For \(E\)-valued forms \(\gamma,\eta\), the pointwise Hermitian inner product \(\langle \gamma,\eta\rangle_{h,\omega}\)  is equivalently defined by
\[
\{\gamma, \star_\omega  \eta\}_h=\gamma^T\wedge h  \overline{\star_\omega  \eta}
= \langle \gamma,\eta\rangle_{h,\omega}\, dV_\omega,
\]
where \(\star_\omega\) is the Hodge star operator of the Kähler metric \(\omega\), and \(dV_\omega = \omega^n/n!\).

The \(L^2\)-inner product is
\[
\langle\!\langle \gamma,\eta \rangle\!\rangle_{h e^{-\phi},\omega}
:= \int_U \langle \gamma,\eta\rangle_{h,\omega}\, e^{-\phi}\, dV_\omega.
\]

We denote by \(\langle\!\langle \cdot,\cdot \rangle\!\rangle_0\) the unweighted global inner product (with \(\phi=0\) and \(h=\mathrm{Id}\)), and by \(\dbar_0^*\) the formal adjoint of \(\dbar\) with respect to \(\langle\!\langle \cdot,\cdot \rangle\!\rangle_0\).

Since $h^{-1}$ is locally bounded,  we can choose a mollifier \(\rho\), set \(\rho_\delta(y)=\delta^{-2n}\rho(y/\delta)\), and for \(U\Subset D\) define
\[
h_\delta^{-1}(z) := \int_{\mathbb C^n} \rho_\delta(y)\, h^{-1}(z-y)\, dV(y),
\qquad
h_\delta(z) := \bigl({h_\delta^{-1}(z)}\bigr)^{-1}.
\]
It is a smooth Hermitian metric on \(U\) for $\delta$ small enough.

\begin{lemma}\label{lem:Bochner_ineq}
Assume that \(h\) satisfies the coarse \(L^2\)-estimate property for \((1,n)\)-forms of constant $C$, then for every \(\alpha\in \operatorname{Dom}(\dbar^*_{h e^{-\phi}})\cap L^2_{(1,n)}(U,E,h e^{-\phi})\), we have
\[
\langle\!\langle B_\phi \alpha,\alpha \rangle\!\rangle_{h e^{-\phi}}
\le
C\,\|\dbar^*_{h e^{-\phi}}\alpha\|_{h e^{-\phi}}^2.
\]
\end{lemma}

\begin{proof}
Since every \((1,n)\)-form is automatically \(\dbar\)-closed, the form \( B_\phi\alpha\) is \(\dbar\)-closed. By the coarse \(L^2\)-estimate property, there exists \(u\in L^2_{(1,n-1)}(U,E,h e^{-\phi})\) such that \(\dbar u =  B_\phi\alpha\) and
\[
\|u\|_{h e^{-\phi}}^2 \le
  C \langle\!\langle B_\phi^{-1}B_\phi\alpha,  B_\phi\alpha \rangle\!\rangle_{h e^{-\phi}}
= C \langle\!\langle B_\phi\alpha,\alpha \rangle\!\rangle_{h e^{-\phi}}.
\]
Using the adjoint relation, we obtain
\[
\begin{aligned}
\langle\!\langle B_\phi\alpha,\alpha \rangle\!\rangle_{h e^{-\phi}}^2
&= \langle\!\langle \dbar u,\alpha \rangle\!\rangle_{h e^{-\phi}}^2
= \langle\!\langle u,\dbar^*_{h e^{-\phi}}\alpha \rangle\!\rangle_{h e^{-\phi}} \\
&\le \|u\|_{h e^{-\phi}}^2\, \|\dbar^*_{h e^{-\phi}}\alpha\|_{h e^{-\phi}}^2 \\
&\le C   \langle\!\langle B_\phi\alpha,\alpha \rangle\!\rangle_{h e^{-\phi}}
\|\dbar^*_{h e^{-\phi}}\alpha\|_{h e^{-\phi}}^2.
\end{aligned}
\]
Hence
\[
\langle\!\langle B_\phi\alpha,\alpha \rangle\!\rangle_{h e^{-\phi}}
\le
C\,\|\dbar^*_{h e^{-\phi}}\alpha\|_{h e^{-\phi}}^2.
\]
\end{proof}

\begin{lemma}\label{lem:adj}
For any \(\beta\in C_c^\infty(U,\Lambda^{1,n}\otimes E)\), set
\[
\eta := e^{-\phi} h_\delta^T \beta,\qquad
\alpha := e^{\phi} h^{-T} \eta.
\]
Then
\begin{align*}
\dbar^*_{h e^{-\phi}} \alpha  = e^{\phi} h^{-T} \dbar_0^* \eta,  \qquad
\dbar^*_{h_\delta e^{-\phi}} \beta  = e^{\phi} h_\delta^{-T} \dbar_0^* \eta.
\end{align*}
\end{lemma}

\begin{proof}
For any compactly supported \((1,n-1)\)-form \(u\), using \eqref{eq h} and the definition of the adjoint,
\begin{align*}
\langle\!\langle \alpha, \dbar u  \rangle\!\rangle_{h e^{-\phi},\omega}
&= \int_U \langle  e^{\phi} h^{-T} \eta, \dbar u \rangle_{h,\omega} e^{-\phi}\, dV_\omega \\
&= \int_U \langle\eta, \dbar u \rangle_{ {\mathbb C^r},\omega}\, dV_\omega \\
&= \int_U \langle \dbar_0^*\eta,   u \rangle_{ {\mathbb C^r},\omega}\, dV_\omega\\
&= \int_U \langle  e^{\phi} h^{-T} \dbar_0^*\eta, u\rangle_{h,\omega} e^{-\phi}\, dV_\omega\\
&= \langle\!\langle e^{\phi} h^{-T} \dbar_0^*\eta, u \rangle\!\rangle_{h e^{-\phi},\omega} .
\end{align*}
Hence $\dbar^*_{h e^{-\phi}} \alpha  = e^{\phi} h^{-T} \dbar_0^* \eta$. Replace \(h\) with \(h_\delta\), then $\dbar^*_{h_\delta e^{-\phi}} \beta  = e^{\phi} h_\delta^{-T} \dbar_0^* \eta$.
\end{proof}

\begin{theorem}[=Theorem \ref{MT1}]\label{thm:conv}
If \(h\) satisfies the coarse \(L^2\)-estimate property with constant $C$ for \((1,n)\)-forms, then \(h_\delta\)  also satisfies the coarse \(L^2\)-estimate property with the same constant $C$ for \((1,n)\)-forms  on every \(U\Subset D\) for  \(\delta\) small enough.
\end{theorem}

\begin{proof}
For any \(\beta\in C_c^\infty(U,\Lambda^{1,n}\otimes E)\), define \(\eta,\alpha\) as in Lemma \ref{lem:adj}. Applying the coarse \(L^2\)-estimate property to \(\alpha\), by Lemma \ref{lem:Bochner_ineq}
\[
\langle\!\langle B_\phi\alpha,\alpha \rangle\!\rangle_{h e^{-\phi}}
\le
C\,\|\dbar^*_{h e^{-\phi}}\alpha\|_{h e^{-\phi}}^2.
\]
Then by  Lemma \ref{lem:adj},
\begin{equation}\label{eq eta}
 \langle\!\langle B_\phi  e^{\phi} h^{-T} \eta,   \eta \rangle\!\rangle_{0}
\le
C\,\|e^{\phi} h^{-T} \dbar_0^* \eta\|_{h e^{-\phi}}^2= C\,\langle\!\langle   e^{\phi} h^{-T} \dbar_0^* \eta,   \dbar_0^* \eta \rangle\!\rangle_{0}.
\end{equation}

Since  $U$ is relatively compact in $D$,  we can apply \eqref{eq eta} to the data $(U,\omega,\phi,h^y)$ for any $y\in\mathbb C^n$ with $|y|$ small enough, where $h^y(z):=h(z-y)$. Then
\[
\langle\!\langle B_{\phi} e^{\phi} (h^y)^{-T} \eta,  \eta \rangle\!\rangle_{0}
\le
C\,\langle\!\langle e^{\phi} (h^y)^{-T} \dbar_0^*\eta,  \dbar_0^*\eta \rangle\!\rangle_{0}.
\]
Multiplying by $\rho_\delta(y)$, integrating over $y\in\mathbb C^n$,  we get
\[
\int_{\mathbb C^n}\rho_\delta(y)
\langle\!\langle B_{\phi} e^{\phi} (h^y)^{-T} \eta,  \eta \rangle\!\rangle_{0}\,dy
\le
C\int_{\mathbb C^n}\rho_\delta(y)
\langle\!\langle e^{\phi} (h^y)^{-T} \dbar_0^*\eta,  \dbar_0^*\eta \rangle\!\rangle_{0}\,dy.
\]
Then by Fubini's theorem, we obtain that
\[
\langle\!\langle B_{\phi} e^{\phi} (h_\delta)^{-T} \eta,  \eta \rangle\!\rangle_{0}
\le
C\,\langle\!\langle e^{\phi} (h_\delta)^{-T} \dbar_0^*\eta, \dbar_0^*\eta \rangle\!\rangle_{0}.
\]
Since by Lemma \ref{lem:adj}, $\dbar^*_{h_\delta e^{-\phi}} \beta  = e^{\phi} h_\delta^{-T} \dbar_0^* \eta$, we get
 \[
\langle\!\langle B_{\phi}\beta, \beta \rangle\!\rangle_{h_\delta e^{-\phi},\omega}
\le
C\|\dbar^*_{h_\delta e^{-\phi}} \beta  \|^2_{h_\delta e^{-\phi},\omega}.
\]
Since \(h_\delta\) is smooth and \(\omega\) is complete, the Andreotti–Vesentini density principle (see \cite[Lemma 4]{Andreotti-Vesentini}) ensures that the formal adjoint and Hilbert adjoint coincide, yielding the estimate
\[
\langle\!\langle B_\phi \beta, \beta \rangle\!\rangle_{h_\delta e^{-\phi},\omega}
\le
C\,\|\dbar^*_{h_\delta e^{-\phi}} \beta  \|^2_{h_\delta e^{-\phi},\omega}
\]
holds for every \(\beta \in \Dom(\dbar^*_{h_\delta e^{-\phi}})\) with finite right-hand side.

We now apply the standard Hörmander–Demailly \(L^2\)-existence theorem. Let \(f\in L^2_{(1,n)}(U,E,h_\delta e^{-\phi})\) be \(\dbar\)-closed. Consider the linear functional
\[
\ell(\dbar^*_{h_\delta e^{-\phi}}\beta) := \langle\!\langle \beta,  f \rangle\!\rangle_{h_\delta e^{-\phi},\omega},
\qquad \beta\in \Dom(\dbar^*_{h_\delta e^{-\phi}}).
\]
Then we have
\begin{align*}
  |\ell(\dbar^*_{h_\delta e^{-\phi}}\beta)|^2
\le  & \langle\!\langle B_\phi \beta, \beta \rangle\!\rangle_{h_\delta e^{-\phi},\omega}\langle\!\langle B_\phi^{-1} f, f \rangle\!\rangle_{h_\delta e^{-\phi},\omega} \\
  \le  &
C \|\dbar^*_{h_\delta e^{-\phi}}\beta\|_{h_\delta e^{-\phi}}^2 \langle\!\langle B_\phi^{-1} f, f \rangle\!\rangle_{h_\delta e^{-\phi},\omega}.
\end{align*}

Thus \(\ell\) is a bounded linear functional on the dense  subspace
\[
\operatorname{Im}(\dbar^*_{h_\delta e^{-\phi}})\subset L^2_{(1,n-1)}(U,E,h_\delta e^{-\phi}).
\]
By the Hahn–Banach theorem and the Riesz representation theorem, there exists \(u\in L^2_{(1,n-1)}(U,E,h_\delta e^{-\phi})\) such that \(\dbar u=f\) and
\[
\|u\|_{h_\delta e^{-\phi}}^2
\le
C\langle\!\langle B_\phi^{-1} f, f \rangle\!\rangle_{h_\delta e^{-\phi},\omega}.
\]
This is precisely the coarse \(L^2\)-estimate property for \(h_\delta\).
\end{proof}

\section{Solution to Question \ref{conj DNW-Ina}}\label{ss solution to conj DNW-Ina}

In this section, we give an affirmative answer to Question \ref{conj DNW-Ina} using Theorem \ref{thm:conv} together with Propositions \ref{pro nak res nak} and \ref{pro dnak res dnak}.
Recall that a singular Hermitian metric $h$ on a holomorphic vector bundle $E$ is said to be Griffiths semi-positive if   $\log|u|^2_{h^*}$ is plurisubharmonic for every local holomorphic section $u$ of $E^*$. Since upper semi-continuous functions are plurisubharmonic if and only if their restrictions to complex lines are subharmonic, we obtain:

\begin{lemma}\label{lem ch Grif line}
  Let $h$ be a singular Hermitian metric on a trivial holomorphic vector bundle $E$ over a domain $D\subset \CC^n$ such that $|u|^2_{h^*}$ is upper semi-continuous for any local holomorphic section $u$ of $E^*$. Then $h$ is Griffiths semi-positive  if and only if $h|_{L\cap D}$ is Griffiths semi-positive for any complex line $L\subset \CC^n$.
\end{lemma}

The following lemma reformulates \cite[Theorem 1.1 and 1.2]{DNWZ22} for singular Hermitian metrics on trivial bundles over planar domains:

\begin{lemma}[{\cite[Theorem 1.1 and 1.2]{DNWZ22}}]\label{lem DNWZ Grif}
   Let $h$ be a smooth Hermitian metric on a trivial holomorphic vector bundle $E$ over a planar domain $D\subset \CC$. If $h$ satisfies the optimal or multiple coarse \(L^2\)-estimate property  for \((1,1)\)-forms, then $h$ is Griffiths semi-positive.
\end{lemma}

\begin{theorem}
Let $(E,h)$ be a holomorphic vector bundle on a complex manifold, equipped with a lower semi-continuous Hermitian metric $h$. Assume that the norm $|s|^2_{h^*}$ is strongly upper semi-continuous for every holomorphic local section $s$ of $E^*$. If $h$ satisfies the optimal or multiple coarse \(L^2\)-estimate property for \((n,1)\) or \((1,n)\)-forms, then $h$ is Griffiths semi-positive.
\end{theorem}

\begin{proof}
Since Griffiths positivity is local, we may assume that \(E=D\times \mathbb C^r\) is a trivial holomorphic vector bundle over a bounded domain \(D\subset \mathbb C^n\).

We first consider the case where $h$ satisfies the optimal or multiple coarse \(L^2\)-estimate property  for \((1,n)\)-forms. Let $U\Subset D$ be a domain, and let $h_\delta$ and $h_\delta^{-1}$ be as above. By Theorem \ref{thm:conv}, $h_\delta$ also satisfies the corresponding \(L^2\)-estimate property for \((1,n)\)-forms. Since $h_\delta$ is smooth, Proposition \ref{pro dnak res dnak} implies that for any complex line $L$, the restriction $h_\delta|_{U\cap L}$ also satisfies the optimal or multiple coarse \(L^2\)-estimate property for \((1,1)\)-forms. Then, by Lemma \ref{lem DNWZ Grif}, $h_\delta|_{U\cap L}$ is Griffiths semi-positive, which means that the norm $|s|^2_{h_\delta^{-1}}$ is plurisubharmonic for every holomorphic local section $s$ of $E^*$. Since $|s|^2_{h_\delta^{-1}}$ converges to $|s|^2_{h^{-1}}$ in $L^1_{\mathrm{loc}}$, and the norm $|s|^2_{h^{-1}}$ is strongly upper semi-continuous, it follows that $|s|^2_{h^{-1}}$ is plurisubharmonic. Hence $h|_{U\cap L}$ is Griffiths semi-positive. By Lemma \ref{lem ch Grif line}, we conclude that $h$ is Griffiths semi-positive.

\medskip

Now assume that $h$ satisfies the optimal or multiple coarse \(L^2\)-estimate property for \((n,1)\)-forms. By Lemma \ref{lem coherent}, the support set $Z$ of the quotient sheaf $\mathcal{O}(E)/\mathcal{E}(h)$ is an analytic subset of $D$. Let $U\Subset D\setminus Z$ be a domain. For any complex line $L$, any holomorphic section $s$ on $U\cap L$ ( $s$ can be viewed as a holomorphic section on $U$) and any disk $\Delta(z)\Subset U\cap L$, it follows from Proposition \ref{pro nak res nak} and the strong upper semi-continuity of $|s|^2_{h^*}$ that there exists a sequence of complex lines $L_k$ parallel to $L$ such that the corresponding  disks $\Delta_k(z_k):=\pi_k(\Delta(z)) \Subset  U\cap L_k$ and
\[
\lim_{k\to\infty} |s|^2_{h^*}(z_k)=|s|^2_{h^*}(z),
\]
and each $h|_{U\cap L_k}$ also satisfies the optimal or multiple coarse \(L^2\)-estimate property for \((1,1)\)-forms.

Let $(h|_{U\cap L_k})_\delta^{-1}$ and $(h|_{U\cap L_k})_\delta$ be as above. By Theorem \ref{thm:conv}, $(h|_{U\cap L_k})_\delta$ also satisfies the optimal or multiple coarse \(L^2\)-estimate property for \((1,1)\)-forms. Then Lemma
\ref{lem DNWZ Grif}  implies that $(h|_{U\cap L_k})_\delta$ is Griffiths semi-positive, and hence $h|_{U\cap L_k}$ is Griffiths semi-positive as well.

By the upper semi-continuity of $|s|^2_{h^*}$, subharmonicity of  $|s|^2_{(\pi_k)^* h^*}$ and Fatou's lemma, we have
\begin{align*}
\frac{1}{ \operatorname{Vol}(\Delta(z))}\int_{\Delta(z)} |s|^2_{h^*}\,d\lambda_1
\ge &
\frac{1}{ \operatorname{Vol}(\Delta(z))} \int_{\Delta(z)} \limsup_{k\to\infty} |s|^2_{(\pi_k)^* h^*}\,d\lambda_1\\
\ge &
\limsup_{k\to\infty} \int_{\Delta(z)} |s|^2_{(\pi_k)^* h^*}\,d\lambda_1\\
\ge &
\limsup_{k\to\infty}   |s|^2_{h^*}(z_k)\\
= &
  |s|^2_{h^*}(z).
 \end{align*}
Thus $|s|^2_{h^*}$ is subharmonic and hence $h|_{U\cap L}$ is Griffiths semi-positive. Then it follows from Lemma \ref{lem ch Grif line}  that $h$ is Griffiths semi-positive on $D\setminus Z$. Finally, since the norm $|s|^2_{h^*}$ is strongly upper semi-continuous, we obtain that $h$ is Griffiths semi-positive on $D$.
\end{proof}

\noindent\textbf{Data availability} Our manuscript has no associated data.

\noindent\textbf{Declarations \\ Conflict of interest} On behalf of all authors, the corresponding author states that there is no conflict of interest.

\end{document}